\documentclass{amsart}
\usepackage{graphicx}
\usepackage{amssymb}
\usepackage{epstopdf}
\usepackage[all,cmtip]{xy}
\usepackage[hidelinks]{hyperref}
\usepackage{xcolor}
\usepackage{mathrsfs}
\usepackage{enumitem}
\usepackage{stmaryrd}

\newcommand{\fm}{\frak{m}}

\newcommand{\sfA}{\mathsf{A}}

\def\del{\partial}

\def\fm{\mathfrak m}

\def\Hom{\operatorname{Hom}}

\def\Ext{\operatorname{Ext}}

\renewcommand{\mod}[1]{\mathsf{mod}(#1)}

\def\depth{\operatorname{depth}}

\def\im{\operatorname{im}}
\def\ev{\operatorname{ev}}
\def\coker{\operatorname{coker}}
\def\ker{\operatorname{ker}}
\def\id{\operatorname{id}}

\def\pd{\operatorname{pd}}

\def\id{\operatorname{id}}

\renewcommand{\geq}{\geqslant}
\renewcommand{\leq}{\leqslant}

\newcommand{\lra}{\longrightarrow}

\newcommand{\Co}[2]{\operatorname{C}_{#1}(#2)}

\theoremstyle{plain}
\newtheorem{thm}[subsection]{Theorem}          \newtheorem*{thm*}{Theorem}
\newtheorem{prp}[subsection]{Proposition}      \newtheorem*{prp*}{Proposition}
\newtheorem{cor}[subsection]{Corollary}        \newtheorem*{cor*}{Corollary}
\newtheorem{lem}[subsection]{Lemma}            \newtheorem*{lem*}{Lemma}
          \newtheorem*{cnj*}{Conjecture}
            \newtheorem*{fct*}{Fact}

\theoremstyle{definition}
\newtheorem{dfn}[subsection]{Definition}       \newtheorem*{dfn*}{Definition}
     \newtheorem*{con*}{Construction}
     \newtheorem*{funcon*}{Functorial Constructions}
      \newtheorem*{obs*}{Observation}
\newtheorem{rmk}[subsection]{Remark}           \newtheorem*{rmk*}{Remark}
\newtheorem{exa}[subsection]{Example}          \newtheorem*{exa*}{Example}
          \newtheorem*{exr*}{Exercise}
          \newtheorem*{sln*}{Solution}
         \newtheorem*{exe*}{Exercise}
         \newtheorem*{qst*}{Question}
         \newtheorem*{prb*}{Problem}
            \newtheorem{stp*}{Setup}
            \newtheorem*{set*}{Setting}
            \newtheorem*{ntn*}{Notation}

\theoremstyle{plain}

\newenvironment{parts}{\begin{enumerate}[label=\upshape (\alph*)]}{\end{enumerate}}
\newenvironment{eqc}{\begin{enumerate}[label=\upshape (\roman*)]}{\end{enumerate}}

\newcommand{\F}{F}

\newcommand{\Tr}{\mathsf{Tr}}

\newcommand{\Cosyz}[2]{\operatorname{Cosyz}_{#1}(#2)}
\newcommand{\StCosyz}[2]{\operatorname{StCosyz}_{#1}(#2)}

\title{Cosyzygy modules}
\author{Petter Andreas Bergh, David A. Jorgensen, Peder Thompson}
\date{\today}                                           

\address[P. A. Bergh]{Institutt for matematiske fag, NTNU, N-7491 Trondheim, Norway}
\email{petter.bergh@ntnu.no}
\urladdr{https://www.ntnu.edu/employees/petter.bergh}
\address[D. A. Jorgensen]{Department of Mathematics, University of Texas at Arlington, 411 S. Nedderman Drive, Pickard Hall 429, Arlington, TX 76019, U.S.A.}
\email{djorgens@uta.edu}
\urladdr{http://www.uta.edu/faculty/djorgens/}
\address[P. Thompson]{Department of Business and Mathematics, M{\"a}lardalen University, V{\"a}ster{\aa}s, Sweden}
\email{peder.thompson@mdu.se}
\urladdr{https://www.mdu.se/en/malardalen-university/staff?id=ptn04}

\thanks{Part of this work was completed at Centro Internazionale per la Ricerca Matematica in Trento, Italy, and part during a visit of the first two authors to Mälardalen University in Västerås, Sweden.}

\begin{document}
\maketitle

\begin{abstract}
We characterize all first cosyzygy modules of a given finitely generated module over a local ring. We do this by introducing the notion of a canonical cosyzygy module, and showing how other cosyzygy modules can be obtained from this module. As an application, we characterize Gorenstein rings via properties of canonical first cosyzygy modules.
\end{abstract}

\section{Introduction}
Let $R$ be a commutative noetherian ring.  Syzygy modules in a projective resolution of a finitely generated $R$-module are well-behaved in two fundamental ways: they exist and are unique up to stable isomorphism. Cosyzygy modules, which arise by taking cokernels of embeddings into projective modules, are more problematic: such an embedding need not exist, and when it does, different embeddings can produce cokernels that are not stably isomorphic. 
Despite this, many classes of modules do have cosyzygy modules, including torsionless modules. Consequently, modules with partial co-resolutions by finitely generated projective modules have been studied in many contexts \cite{AB69,AR96,DT23,EG81,Hos90,MTT17,MRS18}. 

A principal goal of this paper is to describe all first cosyzygy modules of a given finitely generated module $M$, that is, all finitely generated modules $N$ fitting into an exact sequence $0\to M\to P\to N\to 0$ with $P$ projective. To do this, we introduce a distinguished first cosyzygy module $X$---the \emph{canonical} first cosyzygy module---which is characterized by vanishing of $\Ext_R^1(X,R)$. Every torsionless $R$-module has such a cosyzygy module (Proposition \ref{1st_cosyz_exists}) and it is unique up to stable isomorphism (Corollary \ref{stablyiso}). More importantly, we show in Theorem \ref{allcosyz} that if $R$ is a local ring and $M$ is a module with no free summand, then the canonical first cosyzygy module $X$ (with no free summand) controls all first cosyzygy modules of $M$: Taking the start of a minimal free resolution of $X$,
\[\xymatrix{
F_2\ar[r]^{\del_2} & F_1 \ar[r]^{\del_1} & F_0 \ar[r]^{} & X \ar[r] & 0\;,
}\]
with $M=\coker \del_2$, the first cosyzygy modules of $M$ are precisely the modules $\coker\sigma\del_1$ for some map $\sigma:F_0\to F$ to a free module $F$ satisfying $\im\del_1\cap \ker \sigma=0$. In Remark \ref{Mfreesmd} we also explain the case where $M$ has a nonzero free summand.

We then extend the construction to higher cosyzygy modules. The resulting notion of a canonical $n$\textsuperscript{th} cosyzygy module (Definition \ref{canonical_nth}) is closely tied to the notion of $n$-torsionfreeness studied by Auslander and Bridger \cite{AB69}. Indeed, $M$ has a canonical $n$\textsuperscript{th} cosyzygy module if and only if $M$ is $n$-torsionfree (Proposition \ref{nco_ntf}). Moreover, all $n$\textsuperscript{th} cosyzygy modules fit into an exact sequence with the canonical one (Proposition \ref{canonicalnthbuilds}). An application of this theory is a characterization of Gorenstein rings: In Theorem \ref{Gor_char} we show that $R$ is Gorenstein if and only if the canonical first cosyzygy module $X$ of any torsionless $R$-module $M$ with $\depth M\geq \depth R$ has the property that $X$ is also torsionless with $\depth X\geq \depth R$.

Here is a brief outline of the paper: In Section \ref{sec_abelian}, we formalize the notion of cosyzygy objects in any abelian category with enough projectives, and consider some first results in this setting. In particular, we show how $n$\textsuperscript{th} cosyzygy modules behave with respect to short exact sequences (Proposition \ref{main}).  We then define the notion of a canonical first cosyzygy module over a commutative noetherian ring in Section \ref{sec_canonical1st}, and show how a canonical first cosyzygy can be constructed from any other cosyzygy module (Proposition \ref{stack}). Section \ref{sec_classify} contains our main result characterizing all first cosyzygy modules over local rings (Theorem \ref{allcosyz}). In Section \ref{sec_cosyztorsionless}, we give examples that exhibit the range of possible behaviours of first cosyzygy modules. Finally, in Section \ref{sec_Gor} we extend our definition to consider canonical higher cosyzygy modules and use these to give a characterization of local Gorenstein rings in terms of canonical first cosyzygy modules (Theorem \ref{Gor_char}).

As a brief note on terminology, our usage of ``cosyzygy" follows that of, for example, \cite{Jen90} or \cite{MRS18}, and is not to be confused with usage elsewhere in the literature, where ``cosyzygy" may refer to cokernel modules in an injective resolution.

\section{Cosyzygy objects in abelian categories}\label{sec_abelian}
Let $\sfA$ be an abelian category with enough projectives. (For example, the module category of any ring.) For an $\sfA$-complex $P$ and an integer $n$, denote the degree $n$ cokernel by $\Co{n}{P}=\coker(P_{n+1}\to P_n)$.  Every object $M\in \sfA$ has a projective resolution $P\xrightarrow{\sim} M$, and as is standard, for each $n\geq 1$, the object $\Co{n}{P}$ is called an \emph{$n$\textsuperscript{th} syzygy} of $M$.  In particular, for each $n\geq 1$ there is an exact sequence
\[\xymatrix{
0 \ar[r] & \Co{n}{P} \ar[r] & P_{n-1}\ar[r] & \cdots \ar[r] & P_{0}\ar[r] & M \ar[r] & 0
}\]
with each $P_i$ projective. Moreover, $n$\textsuperscript{th} syzygies of $M$ are unique up to stable isomorphism. That is, given any two $n$\textsuperscript{th} syzygies $K$ and $K'$ of $M$, there exist projectives $Q$ and $Q'$ such that $K\oplus Q\cong K'\oplus Q'$. This is the content of Schanuel's Lemma. The following formulation is really a corollary of the classic statement, the module version of which can be found for example in \cite[Corollary 5.5]{Lam99}. The version for abelian categories can be proven in exactly the same way as the module theory version, since pullbacks exist and preserve kernels in an abelian category.

\begin{lem}[Schanuel's Lemma]\label{schanuel}
Let $M$ be an object in $\sfA$. Suppose that
\[\xymatrix{
0\ar[r] & K\ar[r] & P_{n-1} \ar[r] & \cdots \ar[r] & P_0\ar[r] & M\ar[r]  &0
}\]
and
\[\xymatrix{
0\ar[r] & K'\ar[r] & P'_{n-1} \ar[r] & \cdots \ar[r] & P'_0\ar[r] & M\ar[r] & 0
}\]
are exact sequences where all $P_i$ and $P'_i$ are projective. Then 
$$K\oplus P'_{n-1}\oplus P_{n-2}\oplus P'_{n-3}\oplus \cdots \cong K'\oplus P_{n-1}\oplus P'_{n-2}\oplus P_{n-3}\oplus \cdots.$$
In particular, $K$ and $K'$ are stably isomorphic.\hfill$\square$
\end{lem}

One of the main aims of this paper is to develop a notion related to $n$\textsuperscript{th} syzygies:

\begin{dfn}
Let $M$ and $N$ be objects in $\sfA$ and $n\geq 1$ be an integer. 
\begin{parts}
\item $N$ is an \emph{$n$\textsuperscript{th} cosyzygy of $M$} if there exists an exact sequence
\[\xymatrix{
0 \ar[r] & M \ar[r] & P_{-1}\ar[r] & \cdots \ar[r] & P_{-n}\ar[r] & N \ar[r] & 0
}\]
where each $P_{-i}$ is a projective object in $\sfA$. 
\item $N$ is a \emph{stable $n$\textsuperscript{th} cosyzygy of $M$} if $N$ is an $n$\textsuperscript{th} cosyzygy of $M\oplus Q$ for some projective object $Q$ in $\sfA$. 
\end{parts}
Denote the class of $n$\textsuperscript{th} cosyzygies of $M$ by $\Cosyz{n}{M}$ and the class of stable $n$\textsuperscript{th} cosyzygies of $M$ by $\StCosyz{n}{M}$.
\end{dfn}
For any object $M$ and $n\geq 1$, one has $\Cosyz{n}{M}\subseteq \StCosyz{n}{M}$. However, not every object embeds into a projective one, so cosyzygies need not exist and the classes $\Cosyz{n}{M}$ and $\StCosyz{n}{M}$ may be empty. Cosyzygies do exist in many situations, though: For example, if $M$ is an $n$\textsuperscript{th} syzygy, then both classes are nonempty. 
We first consider an elementary description of stable cosyzygy objects:

\begin{lem}\label{lem_cos_syz}
Let $M$ and $N$ be objects in $\sfA$ and $n\geq 1$ be an integer. The following are equivalent:
\begin{eqc}
\item $N$ is a stable $n$\textsuperscript{th} cosyzygy of $M$.
\item There exists a projective resolution $P\xrightarrow{\sim} N$ for which the $n$\textsuperscript{th} syzygy $\Co{n}{P}$ is stably isomorphic to $M$.
\item For every projective resolution $P\xrightarrow{\sim} N$, one has that the $n$\textsuperscript{th} syzygy $\Co{n}{P}$ is stably isomorphic to $M$.
\end{eqc}
\end{lem}
\begin{proof}
The implication (i)$\Rightarrow$(ii) follows from the definition.

(ii)$\Rightarrow$(iii): Suppose there exists a projective resolution $P\xrightarrow{\sim} N$ for which $\Co{n}{P}$ is stably isomorphic to $M$. 
If $P'\xrightarrow{\sim} N$ is any other projective resolution, Schanuel's Lemma \ref{schanuel} implies that $\Co{n}{P}$ and $\Co{n}{P'}$ are stably isomorphic. Stable isomorphism is an equivalence relation so $M$ and $\Co{n}{P'}$ are stably isomorphic as well.

(iii)$\Rightarrow$(i): Let $P\xrightarrow{\sim} N$ be a projective resolution. By assumption, there are projectives $Q$ and $Q'$ such that $M\oplus Q\cong \Co{n}{P}\oplus Q'$. Now, adding the contractible complex $0\xrightarrow{}Q'\xrightarrow{=}Q'\xrightarrow{}0$ to $P$ in degrees $n$ and $n-1$ produces an exact complex
\[\xymatrix@C=2em{
0 \ar[r] & \Co{n}{P}\oplus Q' \ar[r] & P_{n-1}\oplus Q'\ar[r] & P_{n-2} \ar[r] & \cdots \ar[r] & P_0\ar[r] & N\ar[r] & 0\;.
}\]
It follows from this exact complex, and the fact that $M\oplus Q\cong \Co{n}{P}\oplus Q'$, that $N$ is a stable $n$\textsuperscript{th} cosyzygy of $M$. 
\end{proof}

We also need the following observation about first cosyzygies:
\begin{lem}\label{basics2}
Let $M$ and $N$ be objects in $\sfA$, and $Q$ be a projective object in $\sfA$. Then $N \in \Cosyz{1}{M}$ if and only if $N\oplus Q\in \Cosyz{1}{M}$.
\end{lem}
\begin{proof}
If $N\in \Cosyz{1}{M}$, then there is an exact sequence $0\to M\to P\to N\to 0$ with $P$ projective. Adding the complex $Q\xrightarrow{=}Q$ to this yields an exact sequence $0\to M \to P\oplus Q\to N\oplus Q\to 0$, so $N\oplus Q\in \Cosyz{1}{M}$.

Conversely, suppose that $N\oplus Q\in \Cosyz{1}{M}$. There is a projective object $P$ and a commutative diagram with exact rows and columns:
\[\xymatrix{
&&0\ar[d]&0\ar[d]&\\
&&P'\ar[d]&N\ar[d]&\\
0 \ar[r] & M \ar[r] & P \ar[r]\ar[d] & N\oplus Q \ar[r] \ar[d]& 0\\
&&Q\ar[r]^{=}\ar[d]&Q\ar[d]&\\
&&0&0&
}\]
Then there exists a map $P'\to N$ making the diagram commute, and the Snake Lemma tells us that this map is surjective with kernel $M$. The object $P'$ is projective since the middle column splits, and hence $N\in \Cosyz{1}{M}$.
\end{proof}

Denote projective dimension of $M$ by $\pd M$ and stable isomorphism by $\approx$. 

\begin{prp}\label{main}
Let $M$ be an object in $\sfA$, let $0 \to N' \to N \to N''\to 0$ be an exact sequence in $\sfA$, and let $n\geq 1$ be an integer.
\begin{parts}
\item If $\pd N\leq n$, then $N'\in \StCosyz{n}{M}$ if and only if $N''\in \StCosyz{n+1}{M}$. 
\item If $\pd N''\leq n$, then $N\in \StCosyz{n}{M}$ if and only if $N'\in \StCosyz{n}{M}$.
\item If $\pd N'\leq n-1$, then $N\in \StCosyz{n}{M}$ if and only if $N''\in \StCosyz{n}{M}$.
\end{parts}
Furthermore:
\begin{parts}
\item[\emph{(d)}] If $\pd N\leq n$ and $N'\in \Cosyz{n}{M}$, then $N''\in \Cosyz{n+1}{M}$. 
\item[\emph{(e)}] If $\pd N''\leq n-1$, then $N\in \Cosyz{n}{M}$ if and only if $N'\in \Cosyz{n}{M}$.
\item[\emph{(f)}] If $\pd N'\leq n-1$ and $N''\in \Cosyz{n}{M}$, then $N\in \Cosyz{n}{M}$.
\end{parts}
\end{prp}
\begin{proof}
Let $P'$ and $P''$ be projective resolutions of $N'$ and $N''$, respectively, and extend these by the Horseshoe Lemma to a projective resolution $P$ of $N$ so that there is a commutative diagram with exact rows and columns:
\begin{align}\label{HLdiag}
\xymatrix{
& 0 \ar[d] & 0\ar[d]& & 0 \ar[d] & 0\ar[d]  & \\
0 \ar[r] & K' \ar[r]\ar[d] &P'_{n-1}\ar[r]\ar[d] &\cdots \ar[r] & P'_0 \ar[r]\ar[d] & N' \ar[r]\ar[d] & 0\\
0 \ar[r] & K \ar[r]\ar[d] & P_{n-1}\ar[r]\ar[d] &\cdots \ar[r] & P_0 \ar[r]\ar[d] & N \ar[r]\ar[d] & 0\\
0 \ar[r] & K'' \ar[r]\ar[d] &P''_{n-1} \ar[r]\ar[d] &\cdots \ar[r] & P''_0 \ar[r]\ar[d]  & N'' \ar[r]\ar[d] & 0 \\
&0&0 & &0&0 & 
}
\end{align}

First assume that $\pd N\leq n$. Glue together the left (nonzero) column and bottom (nonzero) row of this diagram at $K''$ to obtain an exact sequence
\begin{equation}\label{glue_np1}
\xymatrix{0\ar[r] & K'\ar[r] & K \ar[r] & P''_{n-1}\ar[r] & \cdots \ar[r]& P''_0\ar[r]& N''\ar[r] & 0.}
\end{equation}
Note that $K$ is projective since $\pd N\leq n$. From \eqref{HLdiag} and \eqref{glue_np1}, we see that $K'$ is an $n$\textsuperscript{th} syzygy of $N'$ and is an $(n+1)$\textsuperscript{st} syzygy of $N''$. Thus Lemma \ref{lem_cos_syz} yields that $N'\in \StCosyz{n}{M}$ if and only if $K'$ is stably isomorphic to $M$, which holds if and only if $N''\in \StCosyz{n+1}{M}$. Thus part (a) holds, and part (d) follows by taking $P'$ to be a projective resolution such that $K'=M$.

For part (b), assume that $\pd N''\leq n$. In this case, $K''$ is projective, and so the left column of \eqref{HLdiag} splits. Thus $K\cong K'\oplus K''$, that is, $K$ and $K'$ are stably isomorphic.  In particular, $K\approx M$ if and only if $K'\approx M$. It now follows from Lemma \ref{lem_cos_syz} that $N\in \StCosyz{n}{M}$ if and only if $N'\in \StCosyz{n}{M}$. 

Next, for part (e), suppose that in fact $\pd N''\leq n-1$. If $N'\in \Cosyz{n}{M}$, then in \eqref{HLdiag} we may choose $P'$ and $P''$ so that $K'=M$ and $K''=0$, and it follows that $N\in \Cosyz{n}{M}$. 
For the converse, assume that $N\in \Cosyz{n}{M}$. The case $n=1$ follows from Lemma \ref{basics2} since then $N''$ is projective and $N\cong N'\oplus N''$. Now let $n\geq 2$.
Let $P\xrightarrow{\sim} N$ be a projective resolution such that $\Co{n}{P}=M$. In particular, $K= \Co{1}{P}\in \Cosyz{n-1}{M}$, and there is a commutative diagram with exact rows:
\[\xymatrix{
0 \ar[r] & K \ar[r]\ar[d] & P_0 \ar[r] \ar[d]^{=}& N \ar[r]\ar[d] & 0\\
0 \ar[r] & K''\ar[r] & P_0\ar[r] & N''\ar[r] & 0
}\]
The Snake Lemma shows that the map $K\to K''$ is an injection with cokernel $N'$, i.e., there is an exact sequence $0\to K\to K''\to N'\to 0$. Now apply part (d) to this exact sequence: $\pd K''\leq n-2\leq n-1$ and $K\in \Cosyz{n-1}{M}$, so (d) yields $N'\in \Cosyz{n}{M}$ as desired.

Finally, assume that $\pd N'\leq n-1$. In this case, we may assume that the projective resolution $P'$ of $N'$ was chosen so that $P'_n=0$, that is, $K'=0$ in \eqref{HLdiag}. Thus $K\cong K''$, and it again follows from Lemma \ref{lem_cos_syz} that $N\in \StCosyz{n}{M}$ if and only if $N''\in \StCosyz{n}{M}$, so (c) holds. In particular, if $N''\in \Cosyz{n}{M}$, then we can choose $P''$ so that $K''=M$, and so $N\in \Cosyz{n}{M}$, hence (f) follows.
\end{proof}

\begin{rmk}
In the setting of Proposition \ref{main}, the converses of the statements in (d) and (f) do not hold in general. 
For example, consider the local ring $R=k\llbracket x,y\rrbracket$ where $k$ is a field, and define $L$ by the exact sequence $0\to R\xrightarrow{x}R\xrightarrow{\pi}L\to 0$.
\begin{enumerate}
\item There is an exact sequence $0\to L\oplus L\xrightarrow{\tiny \begin{pmatrix}y & 0 \\ 0 & 1 \end{pmatrix}} L\oplus L\to k\to 0$.
Observe that $\pd L\oplus L=1$ and $k\in \Cosyz{2}{R}$, but that $L\oplus L\not\in\Cosyz{1}{R}$: Indeed, the minimal projective resolution $0\to R^2 \xrightarrow{x} R^2 \to L\oplus L \to 0$ shows that if $a$ is an integer such that $L\oplus L\in \Cosyz{1}{R^a}$, then $a\geq 2$.
\item There is an exact sequence $0\to R \xrightarrow{\tiny\begin{pmatrix}x \\ 0\end{pmatrix}} R\oplus L \xrightarrow{\tiny \begin{pmatrix}\pi & 0 \\ 0 & 1 \end{pmatrix}} L\oplus L \to 0$. As before, $L\oplus L\not\in \Cosyz{1}{R}$, even though $R\oplus L \in \Cosyz{1}{R}$.
\end{enumerate}
\end{rmk}

We now can show that any two (stable) first cosyzygies are related by another (stable) first cosyzygy. Part of the next result also recovers \cite[Remark 2.4]{Cro13}. 

\begin{prp}\label{first}
Let $M$, $N$, and $N'$ be objects in $\sfA$. Then $N,N'\in \StCosyz{1}{M}$ if and only if there exist exact sequences
$$0\to P \to X\to N \to 0 \quad \text{and}\quad 0 \to P'\to X \to N'\to 0$$
where $P$ and $P'$ are projective and $X\in \StCosyz{1}{M}$. 
Moreover, it also holds that if $N,N'\in \Cosyz{1}{M}$, then $X\in \Cosyz{1}{M}$.
\end{prp}
\begin{proof}
First assume that $N$ and $N'$ are stable first cosyzygies of $M$. Then there are exact sequences 
$$0\to M\oplus Q \to P\to N \to 0 \quad \text{and}\quad 0 \to M\oplus Q'\to P' \to N'\to 0$$
with $P,P',Q,Q'$ projective. There is a pushout diagram:
\[\xymatrix{
&0\ar[d]&0\ar[d]&&\\
0 \ar[r] & M\oplus Q\oplus Q' \ar[r] \ar[d] & P\oplus Q' \ar[r] \ar[d]& N \ar[r] \ar[d]^{\cong}& 0\\
0 \ar[r] & P'\oplus Q \ar[r] \ar[d] & X \ar[r] \ar[d]& N \ar[r] & 0\\
& N' \ar[r]^{\cong}\ar[d] & N'\ar[d] &&\\
&0&0&&
}\]
This diagram induces an exact sequence
$$0 \to M\oplus Q \oplus Q' \to P\oplus Q'\oplus P'\oplus Q \to X \to 0$$
showing that $X$ is a stable first cosyzygy of $M$. Moreover, the middle row and middle column give the desired exact sequences. The converse follows directly from Proposition \ref{main}(c) by taking $n=1$.

Moreover, if $N,N'\in \Cosyz{1}{M}$, then $X\in \Cosyz{1}{M}$ by Proposition \ref{main}(f). Alternatively, one can take $Q=0=Q'$ above to see that $X\in \Cosyz{1}{M}$. 
\end{proof}

\begin{rmk}\label{fpd}
Let $Q$ be a projective object and $n\geq 1$. The objects in $\StCosyz{n}{Q}$ are precisely those objects $L\in \sfA$ having projective dimension $\pd L\leq n$. Indeed, if $L\in \StCosyz{n}{Q}$, then $\pd L\leq n$ by definition; for the converse, take a projective resolution $P\xrightarrow{\sim} L$ of length $\pd L$ and add to it the contractible complex $Q\xrightarrow{=}Q$ in degrees $n$ and $n-1$ to see that $L\in \StCosyz{n}{Q}$.
\end{rmk}

\section{Canonical first cosyzygy modules}\label{sec_canonical1st}
Throughout the remainder of this paper, let $R$ be a commutative noetherian ring. We apply the ideas of Section \ref{sec_abelian} to the category $\sfA=\mod{R}$ of finitely generated $R$-modules. In this paper, by ``$R$-module" we always mean a \emph{finitely generated} $R$-module.

In light of Remark \ref{fpd}, given an $R$-module $M$ having a stable first cosyzygy $N$, any module of the form $N\oplus L$ for an $R$-module $L$ with $\pd_RL\leq 1$ is also a stable first cosyzygy of $M$. Thus it is too much to ask to classify all stable first cosyzygies, and so we instead turn our focus to (non-stable) cosyzygy modules. As a starting point, we study first cosyzygy modules, and introduce the following:

\begin{dfn}
Let $M$ and $X$ be $R$-modules such that $X$ is a first cosyzygy of $M$. We say $X$ is \emph{canonical} if $\Ext_R^1(X,R)=0$.
\end{dfn}

We recall some preliminaries: Let $M$ be an $R$-module. Denote the $R$-dual of $M$ by $M^*=\Hom_R(M,R)$. There is a canonical map $\varphi:M\to M^{**}$ given by $\varphi(m)=\ev_m$, where for $f\in M^*$, the evaluation map is defined as $\ev_m(f)=f(m)$. As in \cite{Bas60}, say that $M$ is \emph{torsionless} if $\varphi$ is an injection, and \emph{reflexive} if $\varphi$ is an isomorphism.  Each torsionless $R$-module embeds into a free $R$-module. This is due to a classic construction \cite{AB69}, and we record the argument here that shows the cokernel of this embedding is in fact a canonical first cosyzygy of $M$.

\begin{prp}\label{1st_cosyz_exists}
Every torsionless $R$-module has a canonical first cosyzygy.
\end{prp}
\begin{proof}
Let $M$ be a torsionless $R$-module.
Let $\pi:P\to M^*$ be a surjection with $P$ (finitely generated) projective. The dual map $\pi^*:M^{**}\to P^*$ is then an injection. Since $M$ is torsionless, the canonical map $\varphi:M\to M^{**}$ is an injection as well.  Thus there is an exact sequence
\begin{equation}
\label{ses}
\xymatrix{
0\ar[r] & M \ar[r]^{\pi^*\varphi} & P^*\ar[r] & X\ar[r] & 0
}
\end{equation}
We claim that $X$ is a canonical first cosyzygy of $M$. Dualizing $\pi^*\varphi$, there is a commutative diagram,
\[\xymatrix{
P\ar[r]^{\pi} \ar[d]^{\psi'}_{\cong} & M^*\ar[d]^{\psi}\\
P^{**}\ar[r]^{\pi^{**}}\ar[d]_{=} & M^{***} \ar[d]^{\varphi^*}\\
P^{**}\ar[r]^{(\pi^{*}\varphi)^*} & M^*
}\]
where $\psi:M^*\to M^{***}$ and $\psi':P\to P^{**}$ are the canonical maps. A diagram chase shows that $\varphi^*\psi=\id_{M^*}$: Indeed, for any $f\in M^*$ and $m\in M$, one has
$$\varphi^*\psi(f)(m)=\varphi^*(\ev_f)(m)=\ev_f\varphi(m)=\ev_f\ev_m=\ev_m(f)=f(m).$$
Surjectivity of $\pi$ thus implies surjectivity of $(\pi^*\varphi)^*$. Now dualizing the exact sequence \eqref{ses} gives an exact sequence
\[
\xymatrix{
0\ar[r] & X^* \ar[r] & P^{**}\ar[r]^{(\pi^*\varphi)^*}  & M^*\ar[r] & \Ext_R^1(X,R)\ar[r] & 0
}\]
and so surjectivity of $(\pi^*\varphi)^*$ implies $\Ext_R^1(X,R)=0$, as desired.
\end{proof}
In fact, an $R$-module $M$ is torsionless if and only if some first cosyzygy exists. Indeed, if a first cosyzygy of $M$ exists, then there is an injection $M\to P$ with $P$ a (finitely generated) projective $R$-module, and the commutative diagram
\[\xymatrix{
M\ar[r] \ar[d]&  M^{**}\ar[d]\\
P\ar[r]^{\cong} & P^{**}
}\]
shows that $M\to M^{**}$ is an injection, thus $M$ is torsionless. 

\begin{rmk}
Let $M$ be a torsionless $R$-module.
\begin{enumerate}
\item If $R$ is a Gorenstein ring, the module $X$ in the proof of the previous result was referred to as a \emph{pushforward} of $M$ in \cite{HJW01}.
\item If $0\to M \to P \to X \to 0$ is an exact sequence of $R$-modules with $P$ projective, then $X$ is a canonical first cosyzygy of $M$ if and only if the map $M\to P$ is a flat pre-envelope. Moreover, $X$ is a canonical first cosyzygy of $M$ with no projective summand if and only if the map $M\to P$ is a flat envelope. See the argument in the proof of \cite[Proposition 6.6.8]{EJ00}.
\end{enumerate}
\end{rmk}

Canonical first cosyzygies satisfy a mapping property, in the following sense.

\begin{prp} \label{relation1cosyz}
Let $M$ be a torsionless $R$-module, and $N$ and $X$ be first cosyzygies of $M$, with $X$ canonical. Then there exists a short exact sequence of $R$-modules
$$
0\lra P \lra X\oplus Q \lra N\lra0
$$
where $P$ and $Q$ are projective.
\end{prp}

\begin{proof}
As $N,X\in\Cosyz{1}{M}$, Proposition \ref{first} yields a first cosyzygy $X'$ of $M$ and exact sequences $0\to P\to X'\to N\to 0$ and $0\to Q\to X'\to X\to 0$ of $R$-modules, where $P$ and $Q$ are projective.  Since $X$ is canonical, the second sequence splits, so $X'\cong X\oplus Q$. The first sequence then yields the desired exact sequence.
\end{proof}

As an immediate consequence, we see that canonical first cosyzygies (when they exist) are unique up to stable isomorphism:
\begin{cor}\label{stablyiso} 
Let $M$ be a torsionless $R$-module, and $X$ and $X'$ be canonical first cosyzygies of $M$. Then $X$ and $X'$ are stably isomorphic.
\end{cor}
\begin{proof}
Let $X,X'\in \Cosyz{1}{M}$ be canonical. Proposition \ref{relation1cosyz} yields an exact sequence $0\to P\to X\oplus Q\to X'\to 0$ of $R$-modules, where $P$ and $Q$ are projective. Since $X'$ is canonical, this splits, so $X\oplus Q\cong X'\oplus P$, as claimed.
\end{proof}

The following proposition shows how to explicitly construct a canonical cosyzygy module from any given cosyzygy module. This also provides an alternate proof of Proposition \ref{1st_cosyz_exists}.

\begin{prp}\label{stack} Consider the exact sequence of $R$-modules
\[\xymatrix{
P_2 \ar[r]^{\partial_2} & P_1 \ar[r]^{\partial_1} & P_0 \ar[r]^\epsilon & N \ar[r] & 0
}\]
where each $P_i$ is a (finitely generated) projective $R$-module and $M=\coker\partial_2$, so that $N$ is a first cosyzygy of $M$.
Choose elements $f_1,\dots,f_n\in\Hom_R(P_1,R)$ such that $f_i\partial_2=0$ for each $i=1,\dots,n$. Then there exists a commutative diagram
\[\xymatrix{
&&0\ar[d]&0\ar[d]&\\
&& R^n\ar[d]\ar[r]^{=} & R^n\ar[d] &\\
P_2 \ar[r]^{\partial_2}\ar[d]^{=} & P_1 \ar[r]^-{\tiny\begin{pmatrix}-f_1\\\vdots\\-f_n\\\partial_1\end{pmatrix}}\ar[d]^{=} & R^n\oplus P_0 \ar[r]^-{\epsilon'}\ar[d] & X \ar[r]\ar[d] & 0\\
P_2 \ar[r]^{\partial_2} & P_1 \ar[r]^{\partial_1} & P_0 \ar[r]^\epsilon\ar[d] & N \ar[r]\ar[d] & 0\\
&&0&0&
}\]
with exact rows and columns. Thus $X$ is another first cosyzygy of $M$.  If the cohomology classes 
of $f_1,\dots,f_n$ generate $\Ext^1_R(N,R)$, then $\Ext^1_R(X,R)=0$, and so $X$ is a canonical first cosyzygy of $M$.
\end{prp}

\begin{proof} 
The pushout $X$ of the diagram
\[\xymatrix{
P_1 \ar[r]^{\partial_1}\ar[d]_-{\tiny\begin{pmatrix}f_1\\\vdots\\f_n\end{pmatrix}} & P_0 \\
R^n
}\]
explains exactness of the first row at the modules $R^n\oplus P_0$ and $X$, and exactness at $P_1$ follows from the assumption that $f_i\partial_2=0$. Commutativity of the diagram is then easy to verify. Thus $X$ is a first cosyzygy of $M$.

Now suppose that the cohomology classes of $f_1,\dots,f_n$ generate $\Ext^1_R(N,R)$ and that $g\in\Hom_R(P_1,R)$ satisfies $g\partial_2=0$. Then there exist 
$h\in\Hom_R(P_0,R)$ and $r_i\in R$ such that $g+r_1f_1+\cdots+r_nf_n=h\partial_1$.  In other words, 
for the map ${\tiny\begin{pmatrix}r_1&\cdots&r_n&h\end{pmatrix}}:R^n\oplus P_0\to R$ we have 
$g={\tiny\begin{pmatrix}r_1&\cdots&r_n&h\end{pmatrix}}\partial_1'$ where $\partial_1'={\tiny\begin{pmatrix}-f_1\\\vdots\\-f_n\\\partial_1\end{pmatrix}}$.  This shows that the cycle $g$ is also a boundary, and we see that $\Ext^1_R(X,R)=0$.
\end{proof}

We illustrate the proposition with some examples. The first shows that cosyzygies need not be stably isomorphic.

\begin{exa}\label{nonGor_canonical}
Consider the ring $R=k[x,y]/(x^2,xy)$, where $k$ is a field. There is an exact sequence
\[\xymatrix{
R\ar[r]^{x} & R \ar[r]^{y} & R \ar[r] & R/(y) \ar[r] & 0\;.
}\]
Then $R/(y)$ is a non-canonical first cosyzygy of $R/(x)$, as $\Ext_R^1(R/(y),R)\cong k$ (it is generated by $x$). Since $x^2=0$, we can follow Proposition \ref{stack} and ``stack" this generator on top of $y$ to get another exact sequence:
\[\xymatrix{
R\ar[r]^{x} & R \ar[r]^{\tiny\begin{pmatrix}x \\ y\end{pmatrix}} & R^2 \ar[r] & X \ar[r] & 0
}\]
where $X$ is a canonical first cosyzygy of $R/(x)$. Observe that $X$ is not stably isomorphic to $R/(y)$, since $\Ext_R^1(X,R)=0$ but $\Ext_R^1(R/(y),R)\not=0$ (and stable isomorphism is preserved by $\Ext_R^1(-,R)$).
\end{exa}

The next example gives another illustration of using the ``stacking" procedure in Proposition \ref{stack} to construct a canonical first cosyzygy.
\begin{exa}\label{example1}
Consider the ring $R=k[x,y]/(x^2,xy,y^2)$, where $k$ is a field. Set $M=(x,y)$, the maximal ideal. We consider the first cosyzygies of $M$. Of course one obvious cosyzygy of $M$ is $k$ itself:
\[\xymatrix{
0\ar[r] &  (x,y) \ar[r] & R \ar[r] & k\ar[r] & 0\;.
}\]
One can verify that $\Ext_R^1(k,R)\cong k^3$.
To construct another cosyzygy module, we find minimal generators of $\Ext_R^1(k,R)$.
Starting with a resolution of $k$, these are given by cycles $f:R^2\to R$ making the following diagram commute:
\[\xymatrix{
\cdots \ar[r] & R^4 \ar[rr]^{\tiny\begin{pmatrix}x&y&0&0\\0&0&x&y\end{pmatrix}} \ar[d]&& R^2\ar[r]^{\tiny\begin{pmatrix}x&y\end{pmatrix}}\ar[d]^f & R \ar[r]\ar[d] & k \ar[r]\ar[d]^= & 0\\
\cdots \ar[r] & 0 \ar[rr] && R \ar[r] & X_f\ar[r] & k \ar[r] & 0
}\]
where $X_f$ is the pushout of $f$ and ${\tiny\begin{pmatrix}x & y\end{pmatrix}}$. Three generators 
$f_i:R^2\to R$ of $\Ext_R^1(k,R)$, for $i=1,2,3$, are represented by the matrices: 
\[
{\tiny \begin{pmatrix}x & 0\end{pmatrix}}, \quad 
{\tiny \begin{pmatrix}y & 0\end{pmatrix}}, \quad
{\tiny \begin{pmatrix}0 & x\end{pmatrix}}
\]
For example, $X_1=\coker{\tiny\begin{pmatrix}x&0\\x&y\end{pmatrix}}$ is the first cosyzygy if we use Proposition \ref{stack} and stack the generator $f_1={\tiny \begin{pmatrix}x & 0\end{pmatrix}}$. In this case one computes
$\Ext_R^1(X_1,R)\cong k^2$ and that the cycles $f_2, f_3$ are also generators for $\Ext_R^1(X_1,R)$.
Similarly, one has that $X_2=\coker{\tiny\begin{pmatrix}y&0\\x&y\end{pmatrix}}$ and $X_3=\coker{\tiny\begin{pmatrix}0&x\\x&y\end{pmatrix}}$ are two other choices of non-isomorphic first cosyzygies of $M$.

Continuing to stack generators as Proposition \ref{stack}, one obtains the modules $X_{12}=\coker{\tiny\begin{pmatrix}y&0\\x&0\\x&y\end{pmatrix}}$, $X_{13}=\coker{\tiny\begin{pmatrix}0&x\\x&0\\x&y\end{pmatrix}}$ and $X_{23}=\coker{\tiny\begin{pmatrix}y&0\\0&x\\x&y\end{pmatrix}}$; these are three other choices of non-isomorphic first cosyzygies. Finally, using all of $f_1$, $f_2$, and $f_3$, we see that
\[
X=\coker{\tiny\begin{pmatrix}0&x\\y&0\\x&0\\x&y\end{pmatrix}}\cong
\coker{\tiny\begin{pmatrix}x&0\\y&0\\0&x\\0&y\end{pmatrix}}
\]
is a canonical first cosyzygy of $M$. 
\end{exa}

The last example in this section shows that stacking minimal generators of Ext can create free summands in the resulting first cosyzygy. It also shows that---even over a Gorenstein ring---one may have infinitely many non-isomorphic first cosyzygies (without free summands).

\begin{exa}\label{qc} Consider the ring $R=k[x,y,z]/(x^2-yz)$, where $k$ is a field. We construct cosyzygy modules of $M=\coker \tiny\setlength\arraycolsep{2pt}\begin{pmatrix}-y&x\\x&-z\end{pmatrix}$. 

The module $N=R/(x,y)$ has a free resolution beginning with
\[
\xymatrixcolsep{16mm}
\xymatrix{
R^2 \ar[r]^-{\tiny\setlength\arraycolsep{2pt}\begin{pmatrix}-y&x\\x&-z\end{pmatrix}} & R^2\ar[r]^{\tiny\setlength\arraycolsep{2pt}\begin{pmatrix}x&y\end{pmatrix}} & R \ar[r] & N \ar[r] & 0
}
\]
so $N$ is a first cosyzygy of $M$. One computes that $\Ext^1_R(N,R) \cong N \cong k[z]$. Also, if for each $n\geq 0$ we set $N_n=\coker{\tiny\begin{pmatrix} z^{n+1}& xz^n\\x&y\end{pmatrix}}$, then there is an exact sequence
\[
\xymatrixcolsep{20mm}
\xymatrix{
R^2 \ar[r]^-{\tiny\setlength\arraycolsep{2pt}\begin{pmatrix}-y&x\\x&-z\end{pmatrix}} & R^2\ar[r]^{\tiny\setlength\arraycolsep{2pt}\begin{pmatrix} z^{n+1}& xz^n\\x&y\end{pmatrix}} & R^2 \ar[r] & N_n \ar[r] & 0
}
\]
Suppose $n\geq 1$. Then $\Ext^1_R(N_n,R)\cong R/(x,y,z^n)$, which is a $k$-vector space of dimension $n$. A $k$-basis of $\Ext^1_R(N_n,R)$ is given by the cycles
${\tiny\begin{pmatrix}z^{i+1} & xz^i\end{pmatrix}}:R^2\to R$, $i=0,\dots,n-1$. If we stack the minimal generator ${\tiny\begin{pmatrix}z & x\end{pmatrix}}:R^2\to R$ 
we obtain $N_n'=\coker{\tiny\begin{pmatrix} z&x\\z^{n+1}&xz^n\\x&y\end{pmatrix}}$ which has a free summand, even though $N_n$ has no free summand. 

On the other hand, if $n=0$, then $\Ext_R^1(N_0,R)=0$ so $N_0$ is the canonical cosyzygy of $M$ with no free summand.

Note also that the sequence
\[
\xymatrixcolsep{16mm}
\xymatrix{
R^2 \ar[r]^-{\tiny\setlength\arraycolsep{2pt}\begin{pmatrix}-y&x\\x&-z\end{pmatrix}} & R^2\ar[r]^{\tiny\setlength\arraycolsep{2pt}\begin{pmatrix}z&x\end{pmatrix}} & R \ar[r] & N'' \ar[r] & 0
}
\]
is exact, and so is the beginning of a free resolution of $N''=R/(z,x)$. One computes that
$\Ext_R^1(N'',R)\cong N''\cong k[y]$. However, $N\ncong N''$.

In fact, for any nonzero $a\in R$, one can check that the following is exact:
\[
\xymatrixcolsep{16mm}
\xymatrix{
R^2 \ar[r]^-{\tiny\setlength\arraycolsep{2pt}\begin{pmatrix}-y&x\\x&-z\end{pmatrix}} & R^2\ar[r]^{\tiny\setlength\arraycolsep{2pt}\begin{pmatrix}ax&ay\end{pmatrix}} & R \ar[r] & N_a \ar[r] & 0
}
\]
yielding infinitely many non-isomorphic first cosyzygy modules of $M$.
\end{exa}

\section{Characterizing first cosyzygy modules}\label{sec_classify}
Assume that $R$ is a local ring in this section.  In this case projective modules are free.   As before, by ``$R$-module" we mean a finitely generated $R$-module. 

By Corollary \ref{stablyiso} we know that any two canonical first cosyzygy modules are stably isomorphic.  Over a local ring stably isomorphic modules with no free summands are isomorphic \cite[Chapter IV, Corollary 1.4]{Bas68}.  Therefore the canonical first cosyzygy module with no free summand is unique up to isomorphism.

However, as shown in Example \ref{qc}, canonical first cosyzygy modules constructed using Proposition \ref{stack} might contain free summands. The next lemma shows that we may split off any free summand of a first cosyzygy to obtain another with no free summand. If the original first cosyzygy was canonical, so is the new one, and it is therefore \emph{the} canonical first cosyzygy module with no free summand. (By ``no free summand" we of course always mean no nonzero free summand.)

\begin{lem}\label{freesummands} Consider the exact sequence of $R$-modules
\[\xymatrix{
F_2 \ar[r]^{\partial_2} & F_1 \ar[r]^{\partial_1} & F \ar[r]^\epsilon & N \ar[r] & 0
}\]
with $F$, $F_1$, and $F_2$ free. Then $N=X\oplus G$ for a free summand $G$ if and only if $F=F_0\oplus G'$ for free $R$-modules $F_0$ and $G'$ such that there is an isomorphism $\epsilon|_{G'}:G'\to G$ and $\pi_{G'}\partial_1=0$, where $\pi_{G'}:F\to G'$ is the natural projection.
In this case there is a commutative diagram
\[\xymatrix{
F_2 \ar[r]^{\partial_2}\ar[d]^{=} & F_1 \ar[r]^{\pi_{F_0}\partial_1}\ar[d]^{=} & F_0 \ar[r]^{\epsilon|_{F_0}}\ar[d]^\iota & X \ar[r]\ar[d] & 0\\
F_2 \ar[r]^{\partial_2} & F_1 \ar[r]^{\partial_1} & F \ar[r]^\epsilon & N \ar[r] & 0\\
}\]
where $\pi_{F_0}:F\to F_0$ and $\iota:F_0\to F$ are the natural projection and injection.
\end{lem}

\begin{proof} 
First assume that there is a decomposition $N=X\oplus G$ with $G$ free. Lift the natural injection $\iota_G:G\to N$ along the surjection $\epsilon$ to get a map $s:G\to F$ satisfying $\epsilon s=\iota_G$. Set $G'=s(G)$ and $F_0=\ker \pi_G\epsilon$, where $\pi_{G}:N\to G$ is the natural projection, so that $F=F_0\oplus G'$. Now $\epsilon|_{G'}:G'\to G$ is an isomorphism. Letting $\pi_{G'}:F\to G'$ be the natural projection, we get $\im\partial_1=\ker\epsilon\subseteq \ker \pi_G\epsilon=F_0$ and so $\pi_{G'}\partial_1=0$.

Suppose now that $F=F_0\oplus G'$ where $\epsilon:G'\to G$ is an isomorphism and one has $\pi_{G'}\partial_1=0$. We first claim that $\epsilon|_{G'}$ is injective.  Indeed, let $x\in G'$ be such that
$\epsilon(x)=0$.  Then there exists $y\in\F_1$ such that $\partial_1(y)=x$.  Thus
$x=\pi_{G'}(x)=\pi_{G'}\partial_1(y)=0$. It follows now that $\epsilon(G')\subseteq N$ is a free submodule of $N$.  We claim moreover that it is a free summand.  In fact we'll show that $N=\epsilon(G')\oplus\epsilon(F_0)$. Since $\epsilon$ is surjective we have $N=\epsilon(F)=\epsilon(G')+\epsilon(F_0)$. Now suppose that $\epsilon(x)=\epsilon(z)$ for $x\in G'$ and $z\in F_0$. Then $x-z\in\ker\epsilon=\im\partial_1$. Write $x-z=\partial_1(y)$, for some $y\in F_1$. Then $x=\pi_{G'}(x-z)=\pi_{G'}\partial_1(y)=0$. Thus
$\epsilon(x)=\epsilon(z)=0$.

The commutative diagram follows since $\im\partial_1$ is completely contained in $F_0$.
\end{proof}

Note that two different conditions on cosyzygies arise. Lemma \ref{freesummands} detects a free summand of a cosyzygy by the condition $\pi_{G'}\partial_1=0$. On the other hand, in Proposition \ref{stack} the stacked summand $R^n$ of $R^n\oplus P_0$ satisfies the property that 
$\im \partial_1'\cap R^n=0$, where $\partial_1'={\tiny\begin{pmatrix}-f_1\\\vdots\\-f_n\\\partial_1\end{pmatrix}}$. (Indeed, in the setting of that proposition, if the image of $x\in P_1$ under 
$\partial_1'$ lies in $R^n$, then 
$\partial_1(x)=0$.  Therefore there exists $y\in P_2$ such that $x=\partial_2(y)$.  Therefore
$f_i(x)=f_i\partial_2(y)=0$ for $i=1,\dots,n$, showing that $\partial_1'(x)=0$.)
Together, these give a means of understanding \emph{all} first cosyzygy modules.
 
\begin{thm}\label{allcosyz} 
Assume that $R$ is local and $M$ is a torsionless $R$-module with no free summand. Let $X$ be the canonical first cosyzygy module of $M$ with no free summand and
\begin{equation}\label{Xseq}\xymatrix{
F_2 \ar[r]^{\partial_2} & F_1 \ar[r]^{\partial_1} & F_0 \ar[r]^\epsilon & X \ar[r] & 0
}\end{equation}
a minimal exact sequence of $R$-modules with $M=\coker\partial_2$ and each $F_i$ free. Then an $R$-module $N$ is a first cosyzygy module of $M$ if and only if $N\cong\coker\sigma\partial_1$ for some map $\sigma:F_0\to F$ to a free module $F$ with $\im\partial_1\cap \ker\sigma=0$.
\end{thm}

\begin{proof} 
First note that the first syzygy in a minimal free resolution of $X$ is isomorphic to $M$ by Lemma \ref{lem_cos_syz} and \cite[Chapter IV, Corollary 1.4]{Bas68}; thus such a minimal sequence as in \eqref{Xseq} exists.

Let $N$ be a first cosyzygy of $M$. If $N$ is a canonical first cosyzygy of $M$, then $N$ is stably isomorphic to $X$ by Corollary \ref{stablyiso}, and thus $N\cong X\oplus G$ for some free module $G$ since $X$ has no free summand, by \cite[Chapter IV, Corollary 1.4]{Bas68}. Then observe that the natural injection $\sigma:F_0\to F_0\oplus G$ has the desired properties.
Now assume that $N$ is a non-canonical first cosyzygy of $M$, so that $\Ext_R^1(N,R)\not=0$. There exists an exact sequence of $R$-modules
\begin{equation}\label{Nseq}\xymatrix{
F_2 \ar[r]^{\partial_2} & F_1 \ar[r]^{\partial} & F \ar[r] & N \ar[r] & 0
}\end{equation}
where $F$ is another free module. As \eqref{Xseq} was minimal, we may take \eqref{Nseq} minimal also. Choose elements $f_1,\dots,f_n\in\Hom_R(F_1,R)$ whose cohomology classes generate $\Ext^1_R(N,R)$, and set $\partial'={\tiny\begin{pmatrix}-f_1\\\vdots\\-f_n\\\partial\end{pmatrix}}$.
Then we want to show that there exists a commutative diagram with exact rows:
\begin{equation}\label{bigdiag1}
\xymatrixcolsep{3pc}\xymatrix{
F_2 \ar[r]^{\partial_2}\ar[d]_{\cong}^{\varphi_2} & F_1 \ar[r]^{\partial_1}\ar[d]_{\cong}^{\varphi_1} & F_0 \ar[r]\ar[d]^\alpha & X \ar[r]\ar[d]^\beta & 0\\
F_2 \ar[r]^{\partial_2}\ar@{=}[d] & F_1 \ar[r]^{\partial'}\ar@{=}[d] & R^n\oplus F \ar[r] \ar[d]^\pi & X' \ar[r]\ar[d] & 0\\
F_2 \ar[r]^{\partial_2} & F_1 \ar[r]^{\partial} & F \ar[r] & N \ar[r] & 0\\
}
\end{equation}
The bottom two rows come from Proposition \ref{stack}, where $X'$ is canonical. As all entries of $f_i$ belong to the maximal ideal, the middle row is minimal. 
Commutativity and exactness of the top two rows are provided by diagram \eqref{bigdiag2} below as follows: First express the fact that $X'\cong X\oplus G$ for some free module $G$, and lift this map to an isomorphism $\varphi$ of the (minimal) free resolutions, then split off the free summand $G$ using Lemma \ref{freesummands} to get a map $\iota$ as displayed here:
\begin{equation}\label{bigdiag2}
\xymatrixcolsep{3pc}\xymatrix{
F_2\ar[r]^{\partial_2}\ar@{=}[d] & F_1 \ar[r]^{\partial_1}\ar@{=}[d] & F_0 \ar[r]\ar[d]^{\iota_0} & X \ar[r]\ar[d]^{\iota_{-1}}  & 0 \\
F_2\ar[r]^{\partial_2} \ar[d]^{\varphi_{2}}_{\cong} & F_1 \ar[r]^{\tiny\begin{pmatrix}\partial_1 \\ 0\end{pmatrix}} \ar[d]^{\varphi_{1}}_{\cong} & F_0\oplus G \ar[r]\ar[d]^{\varphi_{0}}_{\cong} & X\oplus G\ar[r]\ar[d]^{\varphi_{-1}}_{\cong} & 0 \\
F_2\ar[r]^{\partial_2} & F_1 \ar[r]^{\partial'} & R^n\oplus F \ar[r] & X' \ar[r] & 0
}
\end{equation}
Set $\alpha=\varphi_0\iota_0$ and $\beta=\varphi_{-1}\iota_{-1}$ to get the preceding diagram \eqref{bigdiag1}.
Now we let $\sigma=\pi\alpha$. Since $\partial_1$ and $\partial$ have the same kernel and the diagram commutes, it follows that 
$\im\partial_1\cap\ker\sigma =0$.  Finally, we see from \eqref{bigdiag1} that $\im\partial=\im\partial\varphi_1=\im \sigma\partial_1$, and so it follows that $N=\coker \partial=\coker\sigma\partial_1$.

On the other hand, suppose that $\sigma:F_0\to F$ is a map with $F$ free such that $\im\partial_1\cap\ker\sigma=0$. Then $\ker\partial_1=\ker\sigma\partial_1$ and so $N\cong \coker\sigma\partial_1$ is a first cosyzygy of $M$.
\end{proof}

\begin{rmk}
The proof of the forward implication of Theorem \ref{allcosyz} above explicitly uses the construction from Proposition \ref{stack}, which is useful for computing examples. Alternatively, this implication can be shown using the vanishing of $\Ext_R^1(X,R)$: First extend the identity maps on $F_i$ for $i\geq 1$ in the minimal free resolutions of $X$ and $N$ to a map $\sigma:F_0\to F$,
\[\xymatrix{
F_2 \ar[r]^{\partial_2} \ar@{=}[d] & F_1\ar[r]^{\partial_1} \ar@{=}[d] & F_0 \ar[r]\ar@{-->}[d]^{\sigma} & X \ar@{-->}[d]\ar[r] & 0\\
F_2 \ar[r]^{\partial_2} & F_1\ar[r]^{\partial} & F \ar[r] & N \ar[r] & 0
}\]
Then observe that $N=\coker \sigma\partial_1$ and since the restriction of $\sigma$ to $\im \partial_1$ is the same as its embedding into $F$, it follows that $\im\partial_1\cap \ker \sigma=0$.
\end{rmk}

We illustrate Theorem \ref{allcosyz} with two examples.

\begin{exa}\label{m2=0}
Consider the ring $R=k[x,y] /(x^2,xy,y^2)$ and module $M=(x,y)$ from Example \ref{example1}. We saw there that the canonical first cosyzygy of $M$ with no free summand is given by 
$X=\coker{\tiny\begin{pmatrix}x&0\\ y&0\\ 0&x \\ 0&y\end{pmatrix}}$. There is a commutative diagram
\[\xymatrix@C=4em{
\cdots \ar[r] & R^4 \ar[rr]^{\tiny\begin{pmatrix}x&y&0&0\\0&0&x&y\end{pmatrix}} \ar[d]^=&& R^2\ar[r]^{\partial_1=\tiny\begin{pmatrix}x&0\\ y&0\\ 0&x \\ 0&y\end{pmatrix}}\ar[d]^= & R^4 \ar[r]\ar[d]^\sigma & X \ar[r]\ar[d] & 0\\
\cdots \ar[r] & R^4 \ar[rr]^{\tiny\begin{pmatrix}x&y&0&0\\0&0&x&y\end{pmatrix}} && R^2 \ar[r]^{\tiny\begin{pmatrix}x & y\end{pmatrix}} & R\ar[r] & k \ar[r] & 0
}\]
where $\sigma$ is given by ${\tiny\begin{pmatrix}1&0&0&1\end{pmatrix}}$. Note that
$\im\partial_1\cap\ker\sigma=0$ and, in this case, the resulting first cosyzygy is $k$.  There are many other choices of $\sigma$, for example, that given by ${\tiny\begin{pmatrix}1&0&0&1\\0&1&0&0\\0&0&1&0\end{pmatrix}}$, which would yield the first cosyzygy 
$N=\coker{\tiny\begin{pmatrix}x&y\\y&0\\0&x\end{pmatrix}}$.
\end{exa}

\begin{exa}
Consider the ring $R=k\llbracket x,y\rrbracket/(xy)$, where $k$ is a field, and the module $M=R/(y)$. We claim that the first cosyzygy modules of $M$ are precisely those isomorphic to a module in the following set:
$$\left\{R/(x^n)\oplus R^m\mid \text{ $n\geq 1$ and $m\geq 0$}\right\}.$$

Set $X=R/(x)$ and note that there is an exact sequence
\[\xymatrix{
\cdots \ar[r]^{x} & R \ar[r]^{y} & R\ar[r]^{x} & R\ar[r] & X \ar[r] & 0\;.
}\]
Thus, as also $\Ext_R^1(X,R)=0$, we see that $X$ is the canonical first cosyzygy module of $M$ with no free summand.
By Theorem \ref{allcosyz}, the first cosyzygy modules of $M$ are precisely the $R$-modules $N$ fitting into the following commutative diagram with exact rows
\[\xymatrix{
\cdots \ar[r]^{x} & R \ar[r]^{y}\ar[d]^{=} & R\ar[r]^{x}\ar[d]^{=} & R\ar[r]\ar[d]^{\sigma} & X \ar[r] \ar[d]& 0\\
\cdots \ar[r]^{x} & R \ar[r]^{y} & R\ar[r]^{\partial} & R^t\ar[r] & N \ar[r] & 0
}\]
such that $(x)\cap \ker \sigma=0$ and $t\geq 1$. The map $\partial=\sigma x$ has the form $\partial={\tiny\begin{pmatrix}r_1\\\vdots\\ r_t\end{pmatrix}}$, where each $r_i\in (x)$. We can express the $r_i$ as power series in $x$, so let $n\geq 1$ be the largest integer such that $r_i\in (x^n)$ for all $i$. After a possible permutation, we may assume that $r_1=x^nu$, where $u$ is a unit of $R$, and hence we may assume that $r_1=x^n$. For $i\geq 2$, write $r_i=x^nq_i$ for some $q_i\in k\llbracket x \rrbracket$, thus row operations convert the defining matrix for $\partial$ into $\partial'={\tiny \begin{pmatrix}x^n \\ 0 \\ \vdots \\ 0\end{pmatrix}}$, and thus $N\cong \coker \partial'=R/(x^n)\oplus R^{t-1}$, where $n\geq 1$ and $t-1=m\geq 0$.
\end{exa}

We end the section by giving a condition on the map $\sigma:F_0\to F$ in Theorem \ref{allcosyz} for two first cosyzygies of $M$ to be isomorphic, and remarking on the assumption in Theorem \ref{allcosyz} that $M$ has no free summand.

\begin{prp} 
Let $M$ be a torsionless $R$-module with no free summand.
Two first cosyzygies $N$ and $N'$ of $M$ are isomorphic if and only if, for the associated maps 
$\sigma:F_0\to F$ and $\sigma':F_0\to F'$ from Theorem \ref{allcosyz}, we have isomorphisms $f:F\to F'$ and $g:F_1\to F_1$ such that $\sigma'\partial_1g=f\sigma\partial_1$.
\end{prp}

\begin{proof}  Assume that $N$ and $N'$ are isomorphic via some isomorphism $\chi:N \to N'$. Then by Theorem \ref{allcosyz} and its proof there exists a diagram of $R$-modules
\[
\xymatrix{
F_2 \ar[r]^{\partial_2} & F_1 \ar[r]^{\sigma\partial_1} & F \ar[r] & N \ar[r]\ar[d]^\chi & 0\\
F_2 \ar[r]^{\partial_2} & F_1 \ar[r]^{\sigma'\partial_1} & F' \ar[r] & N' \ar[r] & 0\\
}\]
where the rows are minimal exact sequences and $F$, $F'$, $F_1$, and $F_2$ are free. Now a standard comparison theorem argument yields isomorphisms $f:F\to F'$ and $g:F_1\to F_1$ such that 
$\sigma'\partial_1g=f\sigma\partial_1$.

Conversely, the existence of isomorphisms $f:F\to F'$ and $g:F_1\to F_1$ satisfying 
$\sigma'\partial_1g=f\sigma\partial_1$ gives a commutative diagram
\[
\xymatrix{
F_1 \ar[d]^g\ar[r]^{\sigma\partial_1} & F \ar[r]\ar[d]^f & N \ar[r] & 0\\
F_1 \ar[r]^{\sigma'\partial_1} & F' \ar[r] & N' \ar[r] & 0\\
}\]
Now the Five Lemma establishes an isomorphism $\chi:N\to N'$.
\end{proof}

\begin{rmk}\label{Mfreesmd}
Understanding first cosyzygy modules of modules without free summands (as assumed in Theorem \ref{allcosyz}) also gives information about first cosyzygy modules of modules having free summands. Indeed, let $M$ be a torsionless $R$-module which decomposes as $M=M'\oplus F$, where $M'$ has no free summand and $F$ is free. If $N$ is a first cosyzygy of $M$, then there is a pushout diagram with exact rows and columns:
\[\xymatrix{
&0\ar[d]&0\ar[d]&&\\
& M' \ar@{=}[r]\ar[d] & M' \ar[d] &&\\
0\ar[r] & M \ar[r]\ar[d] & P \ar[r]\ar[d] & N \ar[r]\ar@{=}[d] & 0\\
0 \ar[r] & F \ar[r]\ar[d] & E \ar[r]\ar[d] & N \ar[r] & 0\\
 & 0 & 0 &&
}\]
In particular, $N$ is a quotient of $E$, a first cosyzygy of $M'$, by a free module.
\end{rmk}

\section{Iterating first cosyzygies}\label{sec_cosyztorsionless}
Now that we have a handle on all first cosyzygies of $M$, we turn our attention to finding first cosyzygies that embed into free modules, so that the process of building the start of a projective co-resolution can continue.  

Recall from \cite{Aus66} (see also \cite{AB69}) that an $R$-module $N$ embeds into a free module if and only if $\Ext^1_R(\Tr N,R)=0$, where $\Tr N$ is the Auslander transpose of $N$. That is, if $N=\coker\partial_1$, where $\partial_1:R^{n_1}\to R^{n_0}$ is a free presentation of $N$, then $\Tr N=\coker\partial_1^T$, where $\partial_1^T:R^{n_0}\to R^{n_1}$ is the transpose of $\partial_1$. In order for $\Ext^1_R(\Tr N,R)$ to vanish it is sometimes advantageous to choose a first cosyzygy $N$ with fewer generators, so that $\partial_1$ has fewer rows. But there are limitations.  For example, if $\partial_1$ consists of only one row then 
$N\cong R/I$, where $I$ is an ideal generated by the entries of $\partial_1$. If this ideal contains a regular element, then $\Hom_R(N,R)=0$ and this implies $\Tr N$ has projective dimension 1, so that 
$\Ext^1_R(\Tr N,R)\ne 0$. 

We illustrate this discussion with an example. 

\begin{exa} \label{qc_local}
Consider the ring $R=k\llbracket x,y,z\rrbracket /(x^2-yz)$, where $k$ is a field, and the $R$-module $M=\coker{\tiny\begin{pmatrix}-y&x\\x&-z\end{pmatrix}}$; this is a local version of the ring in Example \ref{qc}. Then the canonical first cosyzygy $X$ of $M$ with no free summand is shown in the exact sequence 
\[
\xymatrixcolsep{16mm}
\xymatrix{
R^2 \ar[r]^-{\tiny\setlength\arraycolsep{2pt}\begin{pmatrix}-y&x\\x&-z\end{pmatrix}} & R^2\ar[r]^{\tiny\setlength\arraycolsep{2pt}\begin{pmatrix}z&x\\x&y\end{pmatrix}} & R^2 \ar[r] & X \ar[r] & 0
}
\]
If we reduce the first cosyzygy $X$ (via Theorem \ref{allcosyz}) to either of the first cosyzygies $N=R/(x,y)$ or $N'=R/(z,x)$, then the defining ideals contain a regular element, so that $\Ext^1_R(\Tr N,R)\ne 0$ and $\Ext^1_R(\Tr N',R)\ne 0$. It turns out that $\Ext^1_R(\Tr X,R)=0$, so that $X$ itself is already the correct choice for first cosyzygy---that is, it already embeds into a free module.
\end{exa}

We recall an example from \cite{HJS09} where a module $M$ over an artinian ring $R$ is an infinite syzygy but $\Ext^1_R(M,R)\ne 0$.

\begin{exa}\label{HJS} We let $R=k[x_1,x_2,x_3,x_4,x_5]/I$ where $k$ is a field with an element $\alpha\in k$ of infinite multiplicative order and $I$ is an ideal generated by the 11 quadratic homogeneous forms given in \cite[2.4]{HJS09}. It turns out that $R$ is artinian with radical cube zero. The module $M$ defined as the cokernel of the map $d_0$, where the maps $d_i:R^2\to R^2$ are defined by the matrices ${\tiny\begin{pmatrix}x_1&\alpha^ix_2\\x_3&x_4\end{pmatrix}}$, sits in the acyclic complex
\[\xymatrix{
\cdots \ar[r] & R^2 \ar[r]^{d_2} & R^2 \ar[r]^{d_1} & R^2\ar[r]^{d_0} & R^2 \ar[r]^{d_{-1}} & R^2 \ar[r] & \cdots
}
\]
Since ${\tiny \begin{pmatrix} 0 & x_5\end{pmatrix}\begin{pmatrix}x_1&\alpha^ix_2\\x_3&x_4\end{pmatrix}}=0$ for all $i$, all $\Ext^i_R(M,R)$ are nonzero. If we stack this generator as in Proposition \ref{stack} and consider $X$ as the cokernel of the map $R^2\to R^3$ defined by the matrix ${\tiny\begin{pmatrix} 0&x_5\\x_1&x_2\\x_3&x_4\end{pmatrix}}$ we get the canonical first cosyzygy of $\coker d_1$, but this does not embed into a free module.
\end{exa}

Next we look at syzygies over Gorenstein local rings as a source of cosyzygies.
\begin{exa}\label{gor} Let $(R,\fm,k)$ be a local Gorenstein ring of dimension $d$, equivalently, $R$ is local Cohen-Macaulay of dimension $d$ and type 1. Thus 
\[
\Ext^i_R(k,R)=\begin{cases} k\text{ if } i=d\\ 0\text{ if } i\ne d \end{cases}
\] 
Assume $d\geq 1$. It follows that $\Omega^{d-1}(k)$ is a first cosyzygy of $\Omega^d(k)$, but not the canonical one; there are thus at least two non-isomorphic first cosyzygies of $\Omega^d(k)$. The canonical one lives in a complete resolution of $k$ and $\Omega^{d-1}(k)$ lives in the minimal projective resolution of $k$.  In the latter case, the other cosyzygies $\Omega^i(k)$, for $0\le i\le d-2$, are all canonical. One may think of $\Omega^{d-1}(k)$ as the bifurcation point.
\end{exa} 

From these examples and those in the previous two sections, we see that there are examples of both canonical and non-canonical first cosyzygies, that both embed and fail to embed into a free module, over both a Gorenstein and non-Gorenstein local ring---eight possibilities. We enumerate these here:
\begin{enumerate}
\item Over non-Gorenstein rings, Examples \ref{m2=0} and \ref{HJS} give canonical first cosyzygies which do not embed in a free module. 
\item The Mindy ring $k[x,y,z]/(x^2,y^2,yz,z^2)$ is a non-Gorenstein ring with canonical first cosyzygy module $M=(x)$ which does embed in a free module. 
\item The module $N$ in Example \ref{m2=0} is an example of a non-canonical first cosyzygy module which does not embed in a free module, where $R$ is not Gorenstein.
\item The module $k$ in Example \ref{m2=0} is a non-canonical first cosyzygy module which does embed in a free module, over a non-Gorenstein ring.
\item The module $k$ in Example \ref{gor}, for $d\ge 2$, is a canonical first cosyzygy module which does not embed into a free, over a Gorenstein ring.
\item In Example \ref{gor}, for $d\ge 3$, the module $\Omega^1(k)$ is a canonical first cosyzygy which does embed into a free.
\item The module $N$ in Example \ref{qc_local} is a non-canonical first cosyzygy which does not embed into a free, over a Gorenstein ring.
\item The module $\Omega^{d-1}(k)$ in Example \ref{gor} is a non-canonical first cosyzygy which does embed into a free, for $d\ge 2$.
\end{enumerate}

These examples show that not every first cosyzygy is suitable for iteration---that is, for building the start of a projective co-resolution. This motivates the notion of canonical \emph{higher} cosyzygies, which we investigate next and use to characterize Gorenstein rings.

\section{Canonical higher cosyzygy modules and Gorenstein rings}\label{sec_Gor}
Recall that $R$ is a commutative noetherian ring. We aim to characterize Gorenstein rings in terms of properties of cosyzygy modules. To do this, we develop the notion of canonical $n$\textsuperscript{th} cosyzygy $R$-modules, and consider how the existence of such modules is related to the notion of $n$-torsionfreeness. 

Recall from \cite{AB69} that an $R$-module $M$ is $n$-torsionfree if $\Ext_R^i(\Tr M,R)=0$ for $i=1,\ldots , n$. In particular, $1$-torsionfree is the same as torsionless, $2$-torsionfree is the same as reflexive, and for $n\geq 3$, one has that $M$ is $n$-torsionfree if and only if $M$ is reflexive and $\Ext_R^i(M^*,R)=0$ for $i=1,\ldots,n-2$.

\begin{lem}\label{Mntf_iff_Xn1tf}
Suppose there is an exact sequence $0\to M \to P\to X\to 0$ with $P$ projective and $X$ a canonical first cosyzygy of $M$. For $n\geq 2$, $M$ is $n$-torsionfree if and only if $X$ is $(n-1)$-torsionfree.
\end{lem}
\begin{proof}
Since $X$ is a canonical first cosyzygy, $\Ext_R^1(X,R)=0$ so there is also an exact sequence $0\to X^*\to P^*\to M^*\to 0$. Dualizing again yields the lower row is exact in the following commutative diagram, where the vertical maps are canonical:
\[\xymatrix{
0\ar[r] & M \ar[r]\ar[d] & P \ar[r]\ar[d]^{\cong} & X \ar[r]\ar[d] & 0 &\\
0 \ar[r] & M^{**} \ar[r] & P^{**} \ar[r] & X^{**} \ar[r] & \Ext_R^1(M^*,R)\ar[r] & 0
}\]
The Snake Lemma implies that $M$ is torsionless and that the kernel of the canonical map $X\to X^{**}$ is isomorphic to the cokernel of the canonical map $M\to M^{**}$. That is, $M$ is $2$-torsionfree if and only if $X$ is $1$-torsionfree. 

Let $n\geq 3$. Suppose that $M$ is $n$-torsionfree. In particular, $\Ext_R^1(M^*,R)=0$, and so the map $P^{**}\to X^{**}$ in the diagram is surjective. Consequently, reflexivity of $M$ implies the same for $X$ via the Snake Lemma. Moreover, as $X^*$ is a first syzygy module of $M^*$, it follows that $\Ext_R^i(M^*,R)\cong \Ext_R^{i-1}(X^*,R)$ for $i>1$, and it follows that $X$ is $(n-1)$-torsionfree. Conversely, if $X$ is $(n-1)$-torsionfree, then at least $X$ is reflexive so the map $P^{**}\to X^{**}$ is surjective. It follows that $\Ext_R^1(M^*,R)=0$. Again, for $i>1$, the isomorphism $\Ext_R^i(M^*,R)\cong \Ext_R^{i-1}(X^*,R)$ implies the result.
\end{proof}

\begin{rmk}
The assumption that $X$ is canonical in Lemma \ref{Mntf_iff_Xn1tf} is necessary. Indeed, consider the ring $R=k\llbracket x,y\rrbracket /(xy)$, where $k$ is a field. 
The $R$-module $M=\coker {\tiny\begin{pmatrix}y & 0 \\ 0 & x\end{pmatrix}}$ is $n$-torsionfree for all $n\geq 1$, and there is also an exact sequence 
\[\xymatrix{
R^2 \ar[r]^{\tiny\begin{pmatrix}y & 0 \\ 0 & x\end{pmatrix}} & R^2 \ar[r]^{\tiny \begin{pmatrix}x & y\end{pmatrix}} & R \ar[r] & k \ar[r] & 0
}\]
which shows that $k$ is a first cosyzygy of $M$. Moreover, there is an exact sequence $0\to M \to R \to k\to 0$, but $k$ is not even torsionless. Observe that $k$ is not a canonical first cosyzygy module of $M$, as $\Ext_R^1(k,R)\cong k$.
\end{rmk}

\begin{dfn}\label{canonical_nth}
Let $M$ and $X$ be $R$-modules such that $X$ is an $n$\textsuperscript{th} cosyzygy of $M$. We say $X$ is \emph{canonical} if $\Ext_R^i(X,R)=0$ for $i=1,\ldots,n$.
\end{dfn}

\begin{prp}\label{nco_ntf}
Let $M$ be an $R$-module and $n\geq1$ an integer. There exists a canonical $n$\textsuperscript{th} cosyzygy module of $M$ if and only if $M$ is $n$-torsionfree.
\end{prp}

\begin{proof}
The case $n=1$ follows from Proposition \ref{1st_cosyz_exists} and the discussion following it.
We induct on $n\geq 2$. 

First suppose that $M$ has a canonical $n$\textsuperscript{th} cosyzygy module $X$, so there is an exact sequence
\[\xymatrix{
0\ar[r] & M \ar[r] & P_{-1} \ar[r] & \cdots \ar[r] & P_{-n} \ar[r] & X \ar[r] & 0\;.
}\]
Let $N=\coker(M\to P_{-1})$. Evidently $X$ is a canonical $(n-1)$\textsuperscript{st} cosyzygy of $N$. By induction, we get that $N$ is $(n-1)$-torsionfree. Now, since $N$ is also a canonical first cosyzygy of $M$, Lemma \ref{Mntf_iff_Xn1tf} implies that $M$ is $n$-torsionfree, as desired.

Conversely, assume that $M$ is $n$-torsionfree. In particular, $M$ has a canonical first cosyzygy module $X$, i.e., an exact sequence with $P$ (finitely generated) projective:
\[\xymatrix{
0\ar[r] & M \ar[r] & P \ar[r] & X \ar[r] & 0
}\]
and $\Ext_R^1(X,R)=0$. By Lemma \ref{Mntf_iff_Xn1tf}, the module $X$ is $(n-1)$-torsionfree. Induction says that $X$ has a canonical $(n-1)$\textsuperscript{st} cosyzygy module, say $Y$. That is, $\Ext_R^i(Y,R)=0$ for $i=1,\ldots,n-1$.  Since $X$ is an $(n-1)$\textsuperscript{st} syzygy module of $Y$, one also has
$$\Ext_R^n(Y,R)\cong \Ext_R^1(X,R)=0.$$
That is, $Y$ is a canonical $n$\textsuperscript{th} cosyzygy module of $M$.
\end{proof}

\begin{prp}\label{canonicalnthbuilds}
Let $M$ be an $R$-module having a canonical $n$\textsuperscript{th} cosyzygy $X$. For any other $n$\textsuperscript{th} cosyzygy module $N$ of $M$, there is an exact sequence of $R$-modules
\[\xymatrix{
0\ar[r] & K \ar[r] & X\oplus Q \ar[r] & N\ar[r] & 0
}\]
where $Q$ is projective and $\pd_R K\leq n-1$.
\end{prp}
\begin{proof}
The case $n=1$ is already taken care of by Proposition \ref{relation1cosyz}. Assume $n\geq 2$. 
Since $X$ and $N$ are both $n$\textsuperscript{th} cosyzygies of $M$, there are exact sequences
\[\xymatrix{
P_1 \ar[r]^{\del_1} & P_0 \ar[r] & P_{-1} \ar[r] & \cdots \ar[r] & P_{-n} \ar[r] & X \ar[r] & 0
}\]
and 
\[\xymatrix{
P_1 \ar[r]^{\del_1} & P_0 \ar[r] & Q_{-1} \ar[r] & \cdots \ar[r] & Q_{-n} \ar[r] & N \ar[r] & 0
}\]
where $M=\coker\del_1$ and $P_i$ and $Q_i$ are (finitely generated) projective modules. Dualizing both of these sequences yields the following, where the top row is exact since $X$ is a canonical $n$\textsuperscript{th} cosyzygy module of $M$:
\[\xymatrix{
0 \ar[r] & X^* \ar[r] & P_{-n}^* \ar[r] & \cdots \ar[r] & P_{-1}^* \ar[r] & P_0^* \ar[r]^{\del_1^*} & P_1^*\\
& N^* \ar[r] & Q_{-n}^* \ar[r] & \cdots \ar[r] & Q_{-1}^* \ar[r] & P_0^*\ar[r]^{\del_1^*} \ar[u]^{=}& P_1^*\ar[u]^{=}
}\]
(The bottom row need not be exact.) Exactness of the top row and projectivity of each $Q_{-i}^*$ produce maps $\varphi_i:Q_{-i}^*\to P_{-i}^*$ lifting the identity on $M^*$. Dualizing this diagram then produces the following commutative diagram:
\[\xymatrix{
P_1^{**} \ar[r]\ar[d]^{=} & P_0^{**} \ar[r] \ar[d]^{=}& P_{-1}^{**} \ar[r]\ar[d]^{\varphi_{-1}^*} & \cdots \ar[r] & P_{-n}^{**} \ar[r] \ar[d]^{\varphi_{-n}^*}& X^{**} \\
P_1^{**}\ar[r] & P_0^{**} \ar[r] & Q_{-1}^{**} \ar[r] & \cdots \ar[r] & Q_{-n}^{**} \ar[r] & N^{**} 
}\]
In fact both rows are exact as (finitely generated) projective modules are reflexive, and we therefore have a commutative diagram as follows:
\[\xymatrix{
P_1 \ar[r]\ar[d]^{=} & P_0 \ar[r] \ar[d]^{=}& P_{-1} \ar[r]\ar[d]^{(\varphi_{-1}^{*})'} & \cdots \ar[r] & P_{-n} \ar[r] \ar[d]^{(\varphi_{-n}^{*})'}& X\ar[r]\ar@{-->}[d]^{\psi} & 0\\
P_1\ar[r] & P_0 \ar[r] & Q_{-1} \ar[r] & \cdots \ar[r] & Q_{-n} \ar[r] & N\ar[r] & 0
}\]
We can make this a surjective chain map by adding appropriate projective modules to the top row. This gives a new commutative diagram where the vertical maps are all surjective and $Q$ and $P_{-i}'$ are (finitely generated) projective modules:
\[\xymatrix{
P_1 \ar[r]\ar[d]^{=} & P_0 \ar[r] \ar[d]^{=}& P'_{-1} \ar[r]\ar@{->>}[d]^{(\varphi_{-1}^{*})''} & \cdots \ar[r] & P'_{-n} \ar[r] \ar@{->>}[d]^{(\varphi_{-n}^{*})''}& X\ar[r]\ar@{->>}[d]^{\psi'} \oplus Q& 0\\
P_1\ar[r] & P_0 \ar[r] & Q_{-1} \ar[r] & \cdots \ar[r] & Q_{-n} \ar[r] & N\ar[r] & 0
}\]
Evidently then $K=\ker(\psi')$ has projective dimension at most $n-1$.
\end{proof}

\begin{cor}
Let $M$ be an $R$-module with canonical $n$\textsuperscript{th} cosyzygies $X$ and $X'$. Then $X$ and $X'$ are stably isomorphic.
\end{cor} 
\begin{proof}
The short exact sequence in the statement of Proposition \ref{canonicalnthbuilds} splits if $N$ is a canonical $n$\textsuperscript{th} cosyzygy and $\pd_RK\leq n-1$. This follows by dimension shifting.
\end{proof}

We now come to our first characterization of local Gorenstein rings:

\begin{thm}\label{Gor_char}
Assume that $R$ is local.
The ring $R$ is Gorenstein if and only if for every torsionless $R$-module $M$ with $\depth M\geq \depth R$, its canonical first cosyzygy $X$ with no free summand is torsionless and satisfies $\depth X\geq \depth R$.
\end{thm}

\begin{proof}
Set $t=\depth R$.
First suppose that $R$ is Gorenstein. Let $M$ be a torsionless $R$-module with $\depth M\geq t$ (hence $\depth M=t$ in this case). By Proposition \ref{1st_cosyz_exists}, $M$ has a canonical first cosyzygy module $X$ with no free summand; that is, there is a short exact sequence
\[\xymatrix{
0\ar[r] & M\ar[r] & F \ar[r] & X \ar[r]& 0
}\]
with $F$ a (finitely generated) free $R$-module and $\Ext_R^1(X,R)=0$. From this short exact sequence, we also see that $\Ext_R^i(M,R)\cong \Ext_R^{i+1}(X,R)$ for $i\geq 1$.  Since $R$ is Gorenstein, a result due to Ischebeck \cite{Isc69} (see also \cite[Exercise 3.1.24]{BH93}) states that for any (finitely generated) $R$-module $N\ne0$ one has $\depth(R)-\depth(N)=\sup\{i\geq0\mid \Ext_R^i(N,R)\not=0\}$. As $\depth M=t$, we thus have $\Hom_R(M,R)\not=0$ and $0=\Ext_R^i(M,R)\cong\Ext_R^{i+1}(X,R)$ for $i\geq 1$. Therefore, to check that $\depth X\geq t$, it suffices (again by Ischebeck's result) to observe that $\Ext_R^1(X,R)=0$ as $X$ is a canonical first cosyzygy. In fact, $X$ is maximal Cohen-Macaulay, and thus is reflexive (e.g., \cite[Theorem 3.3.10(d)(iii)]{BH93}). In particular, it is torsionless.

Now for the converse. Let $M$ be an arbitrary $(t+2)$\textsuperscript{nd} syzygy module. In particular, it must be torsionless. By the Depth Lemma, we know that $\depth M\geq t$. The assumption then implies that there exists a canonical first cosyzygy module $X_{-1}$ that is torsionless and has $\depth X_{-1}\geq t$: $0\to M\to F_{-1}\to X_{-1}\to 0$. We can continue this process to build an exact sequence
\[\xymatrix{
0\ar[r] & M \ar[r] & F_{-1}\ar[r] & F_{-2}\ar[r] & \cdots \ar[r] & F_{-(t+2)} \ar[r] & X_{-(t+2)}\ar[r] & 0
}\]
where each $F_{-i}$ is (finitely generated) projective and $X_{-i}=\coker (F_{-(i-1)}\to F_{-i})$ is a canonical first cosyzygy module of $X_{-(i-1)}$ for $i>1$. That is, $\Ext_R^1(X_{-i},R)=0$ for $i=1,\ldots,t+2$. This implies that $\Ext_R^i(X_{-(t+2)},R)=0$ for $i=1,\ldots,t+2$, so $X_{-(t+2)}$ is a canonical $(t+2)$\textsuperscript{nd} cosyzygy module of $M$. By Proposition \ref{nco_ntf}, this implies that $M$ is $(t+2)$-torsionfree. This shows that every $(t+2)$-syzygy module is $(t+2)$-torsionfree, thus \cite[Proposition 4.21]{AB69} implies that $R$ is Gorenstein.
\end{proof}

It is known from \cite[Proposition 2.9]{MRS18} that a local ring $R$ is artinian Gorenstein if and only if every $R$-module is torsionless. The next corollary observes that, under the weaker assumption of $\depth R=0$, it is enough for the class of torsionless $R$-modules to be closed under taking canonical first cosyzygies:
\begin{cor}
Assume that $R$ is local. The ring $R$ is artinian Gorenstein if and only if $\depth R=0$ and for every torsionless $R$-module, its canonical first cosyzygy module (with no free summand) is torsionless.\hfill $\square$
\end{cor}

We end the paper with another characterization of Gorenstein local rings:
\begin{thm} 
Assume that $(R,\fm,k)$ is local. Then $R$ is artinian Gorenstein if and only if $k$ has a canonical second cosyzygy that is torsionless. 
\end{thm}

\begin{proof} $(\Longrightarrow)$ is easy.  $(\Longleftarrow)$: Since $k$ is a second syzygy we know from \cite[Lemma 2.8]{MRS18} that $R$ has a one-dimensional socle, generated by $s$, say. Let $x_1,\dots,x_e$ be a minimal generating set for $\mathfrak m$. By assumption we have an exact sequence
\[
\xymatrixcolsep{11mm}
\xymatrix{
R^e\ar[r]^{\tiny\begin{pmatrix}x_1 \cdots x_e\end{pmatrix}} & R \ar[r]^s & R\ar[r]^{\tiny\begin{pmatrix}x_1 \\ \vdots \\ x_e\end{pmatrix}} & R^e\ar[r]^\partial & R^n 
}
\]
From Proposition \ref{stack}, with $\partial_2=s$, we see that the canonical first cosyzygy of $k$ is presented by $s$ and the canonical second cosyzygy of $k$ with no free summand is presented by $\tiny\begin{pmatrix}x_1 \\ \vdots \\ x_e\end{pmatrix}$. We then achieve a commutative diagram
\[
\xymatrixcolsep{11mm}
\xymatrix{
0 \ar[r] & k \ar[r]\ar@{=}[d] & R\ar[r]^{\tiny\begin{pmatrix}x_1 \\ \vdots \\ x_e\end{pmatrix}}\ar@{=}[d] & R^e\ar[r]^{\partial_1^*}\ar@{=}[d] & R^{b_1}\ar[d]^\alpha \\
0 \ar[r] & k \ar[r] & R\ar[r]^{\tiny\begin{pmatrix}x_1 \\ \vdots \\ x_e\end{pmatrix}} & R^e\ar[r]^\partial & R^n 
}
\]
where the top row is the dual of the beginning of the minimal free resolution of $k$ and the map $\alpha$ is induced by dualizing the assumed exact sequence above and comparing with the minimal free resolution, and then dualizing back. Thus we have $\ker\partial_1^*\subseteq\ker(\alpha\partial_1^*)=\ker\partial=\im\tiny\begin{pmatrix}x_1 \\ \vdots \\ x_e\end{pmatrix}$. This shows that 
$\Ext^1_R(k,R)=0$, which means $R$ is injective, i.e. artinian Gorenstein.
\end{proof}

\bibliographystyle{plain}

\end{document}